\documentclass[english]{cccconf}
\usepackage[namelimits]{amsmath} 
\usepackage{amsmath}
\usepackage {amssymb}
\UseRawInputEncoding
\usepackage[comma,numbers,square,sort&compress]{natbib}
\usepackage{epstopdf}
\usepackage{verbatim}
\usepackage{multirow}
\usepackage{color}

\usepackage{dsfont}
\usepackage{float}

\usepackage{amsthm}
\usepackage{graphics} 
\usepackage{epsfig} 
\usepackage{mathptmx} 
\usepackage{times} 

\newtheorem{thm}{Theorem}

\newtheorem{prop}[thm]{Proposition}

\newtheorem{rem}{Remark}

\makeatletter
\newcommand{\biggg}{\bBigg@{3}}

\newcommand{\Biggg}{\bBigg@{3.5}}

\newcommand{\bigggg}{\bBigg@{4}}

\newcommand{\Bigggg}{\bBigg@{4.5}}

\makeatother

\begin{document}

\title{A Robust Fault Detection Filter for Linear Time-Varying System with Non-Gaussian Noise}
\author{Zhemeng Zhang\aref{SWU},
		Yifei Nie\aref{SWU},
        Le Yin\aref{SWU}}


\affiliation[SWU]{School of Computer and Information Sciences,
       Southwest University, Chongqing 400715, P.~R.~China
        \email{zzmmartel@email.swu.edu.cn, nyf20021120@email.swu.edu.cn, yinle0002@swu.edu.cn}}

\maketitle

	\begin{abstract}
	This paper addresses the problem of robust fault detection filtering for linear time-varying (LTV) systems with non-Gaussian noise and additive faults. The conventional generalized likelihood ratio (GLR) method utilizes the Kalman filter, which may exhibit inadequate performance under non-Gaussian noise conditions. To mitigate this issue, a fault detection method employing the $H_{\infty}$ filter is proposed. The $H_{\infty}$ filter is first derived as the solution to a regularized least-squares (RLS) optimization problem, and the effect of faults on the output prediction error is then analyzed. The proposed approach using the $H_{\infty}$ filter demonstrates robustness in non-Gaussian noise environments and significantly improves fault detection performance compared to the original GLR method that employs the Kalman filter. The effectiveness of the proposed approach is illustrated using numerical examples.
	\end{abstract}
	
	
	
	\section{Introduction}

Fault detection is crucial for maintaining system stability, particularly in aerospace and manufacturing industries \cite{zhang2014}. With advancements in modern control systems, researchers have consistently focused on enhancing the safety and reliability of diverse systems \cite{zhang2008a,hwang2009}.

The fault detection approach for linear time-invariant (LTI) systems has been extensively studied in \cite{frank1990,isermann1997}. As nonlinear systems are more commonly encountered in practical applications, various methods have been developed to address nonlinear systems. One such method involves converting a nonlinear system into a linear time-varying (LTV) system using a first-order Taylor series expansion, which is the well-known extended Kalman filter (EKF) approach. However, the accuracy of EKF state estimates is limited by the neglect of higher-order terms \cite{einicke1999}. To overcome this limitation, a statistical linearization method leading to the unscented Kalman filter has been proposed \cite{lefebvre2002}. Compared to the EKF, this approach improves the accuracy of the linearized model. Since nonlinear systems are often treated as linear systems, fault detection for linear systems is of particular interest to researchers.



The Kalman filter is well-known for providing optimal linear state estimation under the assumption of Gaussian noise. A fault detection approach based on the Kalman filter, derived from generalized least squares, was proposed in \cite{chen2000}. The generalized likelihood ratio (GLR) method for fault detection was introduced in \cite{willsky1974,willsky1976} and further explored in \cite{basseville1993}. An extension of the GLR method using a recursive Kalman filter under weaker assumptions was proposed in \cite{zhang2014}. However, these studies primarily focus on models with Gaussian noise.

A fault detection approach for non-Gaussian noise, based on game theory, was introduced in \cite{chung1998}. Subsequent works, such as \cite{zhong2010,li2009}, applied the $H_\infty$ filter for fault detection in LTV systems with non-Gaussian noise. These studies leveraged the $H_\infty$​ filter for its ability to minimize the maximum estimation error between the true and estimated states, offering greater robustness than the Kalman filter \cite{simon2006}. Based on the previous works, combining the GLR method with the $H_\infty$​ filter presents a promising direction for fault detection in non-Gaussian noise systems.

In this paper, the $H_{\infty}$ filter is first reformulated in a RLS form. Then, inspired by the GLR method in \cite{zhang2014}, a generalized innovation ratio (GIR) method is proposed, which derives the fault vector using the least-squares method. The proposed GIR method extends the GLR method for non-Gaussian noise environments, and the newly introduced $H_{\infty}$ filter enhances fault detection performance in such environments.


The paper is organized as follows. Section \ref{sec:kfhf} reformulates the Kalman filter and the $H_{\infty}$ filter in RLS form. Section \ref{sec:innovation} examines the fault effect on the filter innovation and introduces the persistent excitation condition, which can replace the requirement for multiple sensors. Section \ref{sec:theta} presents a least-squares approach for solving the fault vector $\theta$. Section \ref{sec:gir} proposes the GIR method, a fault detection approach for LTV systems with non-Gaussian noise. Section \ref{sec:onset} introduces an fault detection approach for LTV systems with faults of unknown onset time. Finally, numerical examples are given in Section \ref{sec:example} to validate the proposed approach.

\textit{Notation:} $\left( \,\cdot \,\right)^T$ and $\left( \, \cdot \, \right)^{-1}$ are the transposition and inverse of a matrix, respectively. $\mathbb{E}\left[ \,\cdot \,\right]$ represents the expectation operator. $\mathbf{I}$ denotes the identity matrix with appropriate dimension. $\mathbb{R}^n$ stands for the Euclidean space of $n$-dimensional. $\left\{ x_0\,,\,\, x_1 \,, ... \,,\,\,x_k \right\}$ is represented as $x_{0:k}$. Given a vector $x$ and a positive-definite matrix $A \succ 0$, ${\lVert x \rVert}_A^2 = x^TAx$. Given a cost function $J(x)$, $\hat{x} = \arg\min J(x)$ means that $\hat{x}$ is the value at which $J(x)$ attains its minimum.

\section{Preliminaries}
\label{sec:kfhf}
Consider the following state-space model of a fault-free linear time-varying systems with control inputs
\begin{align}
	x_{k+1} &= A_kx_k + B_ku_k + w_k \\
    y_k &= C_kx_k + D_ku_k + v_k
\end{align}
where $x_k \in \mathbb{R}^n$, $u_k \in \mathbb{R}^l$ and $y_k \in \mathbb{R}^p$ denotes the state, input and output vectors at time step $k$, respectively. The process and measurement noises $w_k$ and $v_k$ are mutually independent and independent of $x_0$, and their covariance matrices are given by $Q_k$ and $R_k$, respectively. More specifically, $w_k$ and $v_k$ are assumed to be zero-mean and Gaussian in a typical Kalman filter setting and only zero-mean in a typical $H_\infty$ filter setting. $A_k$, $B_k$, $C_k$, $D_k$ are the state transition matrix, control input matrix, observation matrix and direct feedthrough matrix, respectively.

Then, then LTV systems with additive faults can be formulated as
\begin{align}
	x_{k+1} &= A_kx_k + B_ku_k + w_k + \varPsi_k\theta \label{statefault}\\
    y_k &= C_kx_k + D_ku_k + v_k 
\end{align}
where $\varPsi_k$ is a partially known $n \times m$ time-varying fault profile matrix, $\theta$ is an $n$-dimensional constant fault vector. Different from \cite{zhong2005} that assumes a known constant fault profile $\varPsi$ and a time-varying fault vector $\theta_k$. In this case, the number of sensors $p$ must satisfy $p \geq m$. The above condition can be replaced by a persistent excitation condition shown in \cite{zhang2014}.

\subsection{Typical fault}
Since $\theta$ is a constant in \eqref{statefault}, the fault profile with step fault can be represented as
\begin{equation}
	\varPsi_k(r) = \tilde{\varPsi}_k \times \mathds{1}_{\left\{ k\, \geq \,r \right\}}
\end{equation}
where $\tilde{\varPsi}_k$ is a known fault profile, $r$ is the unknown onset time of step fault, $\varPsi_k(r)$ represents the dependence of $\varPsi_k$ on $r$, and $\mathds{1}_{\left\{ \cdot \right\}}$ denotes the indicator function.

The fault profile induced by the impulsive fault can be represented as
\begin{equation}
	\varPsi_k(r) = \delta_{k\,,\,\, r} \times \mathbf{I}
\end{equation}
where $\delta_{k\,, \,\,r}$ is the Kronecker delta function that equals 1 when $k=r$. 

It has been shown that fault detection can be achieved by comparing the difference between the state estimates obtained from the faulty system and the fault-free system \cite{zhang2014}. Therefore, accurate state estimation using an appropriate filter is particularly crucial in fault detection.

\subsection{Kalman filter}
For systems with Gaussian noises, the Kalman filter is known as the best linear unbiased estimator in the sense that no other linear estimators can provide a smaller estimation error variance \cite{sorenson1970}. The Kalman filter can be expressed as the solution to a batch-form least-squares problem 
\begin{align}
	J(x_{0:k+1})&=\lVert x_0-\hat{x}_{0|-1} \rVert _{P_{0|-1}^{-1}}^{2} \notag \\
	&\quad +\sum_{k=0}^{N-1}{\lVert x_{k+1}-A_kx_k - B_ku_k\rVert _{Q_k^{-1}}^{2}} \notag \\
	&\quad + \sum_{k=0}^N{\lVert y_k-C_kx_k - D_ku_k\rVert _{R_k^{-1}}^{2}}
\end{align}
where $\hat{x}_{k|k-1}$ and $P_{k|k-1}$ are the priori state estimate and its error covariance, respectively, at time step $k$. Alternatively, following \cite{sorenson1970}, an equivalent recursive least-squares cost function for the Kalman filter is given by
\begin{align}
	J(x_k)&=\|  x_k-\hat{x}_{k|k-1} \|_ {P_{k|k-1}^{-1}}^2 +\| y_k-C_kx_k - D_ku_k\|_{ R_{k}^{-1}}^2 
	\label{rfcost}
\end{align}

Let $\tilde{x}_k$ be the priori state estimation error, i.e.,
\begin{equation}
	\tilde{x}_k \triangleq x_k - \hat{x}_{k|k-1}
\end{equation}
Then, the cost function \eqref{rfcost} can be rewritten into the following compact linear regression form 
\begin{equation}
	m_k = H_kx_k + e_k \label{kflin}
\end{equation}
where
\begin{equation}
	m_k = \begin{bmatrix}
		\hat{x}_{k|k-1} \\
		\,\,\, y_k - D_ku_k\,\,\,
	\end{bmatrix},\,\,\,\, H_k = \begin{bmatrix}
		\mathbf{I} \\
		\,\,\,C_k\,\,\,
	\end{bmatrix} ,\,\,\,\, e_k = \begin{bmatrix}
		\,\,\,-\tilde{x}_{k}\,\,\, \\
		v_k
	\end{bmatrix} \notag
\end{equation}
\begin{equation}
	W_k = \mathbb{E}[\,e_ke_k^T\,] = \begin{bmatrix}
		\,\,\,P_{k|k-1} & 0 \,\,\,\\
		0 & R_k \,\,\,
	\end{bmatrix}\notag
\end{equation}
The state estimate $\hat x_{k|k}$ and its associated error covariance $P_{k|k}$ can be obtained as the solution to \eqref{kflin} and are given by
\begin{align}
	\hat{x}_{k|k} &= \hat{x}_{k|k-1} + K_k( y_k - C_k\hat{x}_{k|k-1} - D_ku_k ) \label{lskf1}\\
	P_{k|k} &= ( \mathbf{I} - K_kC_k ) P_{k|k-1} \label{lskf2}
\end{align}\
where
\begin{align}
	K_k &= P_{k|k-1}{C_k}^T\left( C_kP_{k|k-1}{C_k}^T + R_k \right)^{-1} \label{lskf3}
\end{align}

Now, the state can be estimated recursively using \eqref{lskf1} and \eqref{lskf2}. In practical applications, the distribution of the noise may be unknown, and it is necessary to leverage a more robust filtering algorithms to such handle uncertainties.

\subsection{$H_{\infty}$ filter}
$H_{\infty}$ filter is widely applied in practical system as a robust filtering algorithms. $\alpha$ is the performance parameter to limit the ratio of the estimation error to the system error. The cost function is given by
\begin{align} 
	\frac{\sum_{i=0}^{k-1}{\lVert z_i-\hat{z}_{i|i-1} \rVert _{S_i}^{2}}}     
	{\| x_0-\hat{x}_{0|-1} \|^{2}_{P_{0|-1}^{-1}}+\sum_{i=0}^{k-1}{\left( \lVert w_i \rVert _{Q_{i}^{-1}}^{2}+\lVert v_i \rVert _{R_{i}^{-1}}^{2} \right)}}<\frac{1}{\alpha}  \label{eq:cost1}	
\end{align}
where $Q_k$, $R_k$ and $S_k$ are symmetric positive-definite matrices chosen by the designer based on the specific problem. In contrast to $H_{\infty}$ filter, the Kalman filter minimizes the variance of the estimation error weighted by $S_k$ \cite{simon2006}. In cost function \eqref{eq:cost1}, the predicted object becomes a linear combination of the state estimate
\begin{equation}
	z_k = L_k x_k
\end{equation}
where $L_k$ is assumed to be a full-rank matrix chosen by designer. When $L_k = \mathbf{I}$, the states will be estimated directly.

It should be mentioned that the cost of the $H_{\infty}$ filter in \eqref{eq:cost3} has an additional regularization term compared to the cost of the Kalman filter in \eqref{rfcost}. 


\begin{align}
J(x_{k})&=\|x_{k}-\hat x_{k|k-1}\|^2_{P_{k|k-1}^{-1}} -\alpha\| z_k-\hat z_{k|k-1}\|_{S_k}^2 \notag \\
&\quad + \|y_k-C_kx_k - D_ku_k\|_{R_k^{-1}}^2 \label{eq:cost3}
\end{align}

The solution to the regularized least-squares problem in \eqref{eq:cost3} is given by
\begin{align} 
	\hat{x}_{k|k} &= \hat{x}_{k|k-1} + K_k\left( y_k - C_k\hat{x}_{k|k-1} - D_ku_k\right) \\
	P_{k|k} &= { \left( P_{k|k-1}^{-1} - \alpha L_k^TS_kL_k + C_k^TR_k^{-1}C_k \right) }^{-1} \\
	K_k &= { \left( P_{k|k-1}^{-1} - \alpha L_k^TS_kL_k + C_k^TR_k^{-1}C_k \right) }^{-1} C_k^TR_k^{-1} \label{hfkk} \\
	\hat{x}_{k+1|k} &= A_k\hat{x}_{k|k} + B_ku_k \\
	P_{k+1|k} &= A_kP_{k|k}{A_k}^T+Q_k
\end{align}

\begin{rem}
	In practical applications, $S_k$ is usually set as an identity matrix, which indicates that the same interest is assigned to the states in $z_k$. Note that \textcolor{black}{the parameter $\alpha$ governs the trade-off between the robustness and sensitivity of the filter.} When $\alpha \rightarrow 0$, the additive regularization term in cost function \eqref{eq:cost3} vanishes, and the $H_{\infty}$ filter reduces to the Kalman filter.
\end{rem}

Let $\tilde{z}_k = z_k - \hat{z}_{k|k-1}$, then, similar to \eqref{kflin}, the $H_{\infty}$ filtering problem can likewise be formulated into a linear regression form \cite{nie2024}.
\begin{equation}
	\tilde{m}_k = \tilde{H}_kx_k + \tilde{e}_k \label{eq:HLS}
\end{equation}
where
\begin{equation}
	\tilde{m}_k = \begin{bmatrix}
		\hat{x}_{k|k-1} \\
		\hat{z}_{k|k-1} \\
		\,\,\, y_k - D_ku_k\,\,\,
	\end{bmatrix},\,\,\,\, \tilde{H}_k = \begin{bmatrix}
		\mathbf{I} \\
		L_k \\
		\,\,\,C_k\,\,\,
	\end{bmatrix} ,\,\,\,\, \tilde{e}_k = \begin{bmatrix}
		\,\,\,-\tilde{x}_{k}\,\,\, \\
		-\tilde{z}_{k} \\
		v_k
	\end{bmatrix} \notag
\end{equation}
\begin{equation}
	\tilde{W}_k = \mathbb{E}[\,\tilde{e}_k\tilde{e}_k^T\,] = \begin{bmatrix}
		\,\,\, P_{k|k-1} & P_{k|k-1}L_k^T & 0 \,\,\,\\
		L_kP_{k|k-1} & L_kP_{k|k-1}L_k^T & 0 \,\,\,\\
		0 & 0 & R_k\,\,\,
	\end{bmatrix} \notag
\end{equation}
To obtain the recursions of the $H_{\infty}$ filter, a different weight matrix $W_k$ will be used for solving \eqref{eq:HLS}. 
\begin{equation}
	W_k = \begin{bmatrix}
		\,\,\,P_{k|k-1} & 0 & 0 \,\,\,\\
		0 & -\frac{1}{\alpha}S_k^{-1} & 0 \,\,\,\\
		0 & 0 & R_k \,\,\,
	\end{bmatrix}	
\end{equation}
Then, the state estimate and its error covariance can be obtained by solving the linear regression problem as follows
\begin{align}
	\hat x_{k|k}&=( \tilde{H}_{k}^{T}W_{k}^{-1}\tilde{H}_k ) ^{-1}\tilde{H}_{k}^{T}W_{k}^{-1}\tilde{m}_k \label{lrx}\\
	P_{k|k} &= ( \tilde{H}_{k}^{T}W_{k}^{-1}\tilde{H}_k ) ^{-1}\ \label{lrp}
\end{align}
which can be further expressed as
\begin{align}
\hat x_{k|k} &= \left( \mathbf{I} - K_kC_k \right) \hat{x}_{k|k-1} + K_k\left( y_k - D_ku_k \right) \\
	P_{k|k} &= { \left( P_{k|k-1}^{-1} - \alpha L_k^TS_kL_k + C_k^TR_k^{-1}C_k \right) }^{-1} \label{hpk} 
\end{align}
where
\begin{equation}
	K_k = { \left( P_{k|k-1}^{-1} - \alpha L_k^TS_kL_k + C_k^TR_k^{-1}C_k \right) }^{-1} C_k^TR_k^{-1}
\end{equation}




\section{Fault effect on the innovation}
\label{sec:innovation}

To detect the fault, it is necessary to define an indicator for measuring the degree of the fault in the system. $\epsilon_k$ is the output prediction errors (also called innovation). The definition of $\epsilon_k$ is as follows
\begin{equation}
	\epsilon_k \triangleq y_k - C_k\hat{x}_{k|k-1} - D_ku_k
\end{equation}
where $\hat{x}_{k|k-1}$ is the \textcolor{black}{priori} state estimate. Note that $\epsilon_k$ denotes the error between the observation and its prediction. The innovation reflect the degree of observation offset. \textcolor{black}{A larger innovation indicates a greater deviation of the observation from the fault-free system, and vice versa. Therefore, the innovation $\epsilon_k$ can be utilized for fault detection.}

The state prediction error and innovation also can be expressed as
\begin{align}
	\tilde{x}_{k+1} &= A_k\left( \mathbf{I} - K_kC_k \right)\tilde{x}_k - A_kK_kv_k + w_k + \varPsi_k\theta \\
	\epsilon_k &= C_k\tilde{x}_k + v_k \label{exv}
\end{align}
The relationship between the fault vector $\theta$ and innovation $\epsilon_k$ can be obtained as 
\begin{equation}
	\epsilon_k = \epsilon_k^0 + C_k\varGamma_k\theta 
	\label{e0theta}
\end{equation}
where $\epsilon_k^0$ is the innovation of the same $H_{\infty}$ filter applied to the fault-free system. The recursive form of $\varGamma_k$ is as follows
\begin{equation}
	\varGamma_{k+1} = A_k\left( \mathbf{I} - K_kC_k \right)\varGamma_k + \varPsi_k
\end{equation}
From \eqref{exv}, the covariance $\varSigma_k$ of the innovation is
\begin{equation}
	\varSigma_k = \mathbb{E}[\epsilon_k\epsilon_k^T] = C_kP_{k|k-1}C_k^T + R_k
	\label{sigmak}
\end{equation}
where $P_{k|k-1}$ is the covariance matrix of the priori estimation error $\tilde{x}_k$.

Before proceeding further, consider the following condition
\begin{equation}
	\sum_{k = k - s + 1}^{k} \varGamma_k^TC_k^T\varSigma_k^{-1}C_k\varGamma_k \geq \gamma \times \mathbf{I} \label{eq:ineq}
\end{equation}
\textcolor{black}{where there exists an integer $s > 0$ and a real number $\gamma > 0$ such that, for any $k > s$, the inequality \eqref{eq:ineq} always holds.} This condition is called the persistent excitation condition \cite{zhang2014} and will be employed in the ensuing section to ensure that the inverse of \eqref{eq:EK} exists.






\section{Solving the fault vector}
\label{sec:theta}
For solving the fault vector in \eqref{e0theta}, the MLE method is employed in \cite{zhang2014}. By contrast, this paper adopts the least-squares approach to solve for the fault vector. The cost function of \eqref{e0theta} is given by 
\begin{equation}
	J(\theta) = \sum_{j=1}^{k} {\lVert \epsilon_j^0 \rVert}^2_{\varSigma_j^{-1}} = \sum_{j=1}^{k} {\lVert \epsilon_j - C_j\varGamma_j\theta \rVert}^2_{\varSigma_j^{-1}}
	\label{costJ}
\end{equation}
\begin{rem}
	\label{rem2}
	By assumption, $\theta$ is a constant, and $C_k$ and $\varGamma_k$ are known at time step $k$. Hence, the covariance $\varSigma_k^0$ of the innovation $\epsilon_k^0$ in the fault-free system is equal to the covariance $\varSigma_k$ in \eqref{sigmak}. 
\end{rem}
The solution to \eqref{costJ} can be written as
\begin{equation}
	\hat{\theta}_k = \arg\min_{\theta} \sum_{j=1}^{k} {\left( \epsilon_j - C_j\varGamma_j\theta \right)}^T \varSigma_j^{-1} {\left( \epsilon_j - C_j\varGamma_j\theta \right)}
\end{equation}
Let $\frac{dJ}{d\theta} = 0$, it can be obtained that
\begin{equation}
	\sum_{j=1}^{k} \varGamma_j^TC_j^T\varSigma_j^{-1}\epsilon_j = \sum_{j=1}^{k} \varGamma_j^TC_j^T\varSigma_j^{-1}C_j\varGamma_j \theta
\end{equation}
and
\begin{equation}
	\begin{aligned}
	\hat{\theta}_k &= {\left( \sum_{j=1}^{k} \varGamma_j^TC_j^T \varSigma_j^{-1} C_j\varGamma_j \right)}^{-1} \left( \sum_{j=1}^{k} \varGamma_j^TC_j^T \varSigma_j^{-1} \epsilon_j \right) \\
	&= E_k^{-1}d_k
	\end{aligned}
	\label{thetak}	
\end{equation}
where
\begin{align}
	E_k &\triangleq \sum_{j=1}^{k} \varGamma_j^TC_j^T \varSigma_j^{-1} C_j\varGamma_j \label{ek} \\
	d_k &\triangleq \sum_{j=1}^{k} \varGamma_j^TC_j^T \varSigma_j^{-1} \epsilon_j \label{dk}
\end{align}
Then, the recursive forms of $E_k$ and $d_k$ can be expressed as
\begin{align}
	E_k &= E_{k-1} + \varGamma_k^TC_k^T \varSigma_k^{-1} C_k\varGamma_k \label{eq:EK} \\
	d_k &= d_{k-1} + \varGamma_k^TC_k^T \varSigma_k^{-1} \epsilon_k
\end{align}

\section{Fault detection}
\subsection{Fault effect on generalized innovation ratio}
\label{sec:gir}
Consider the GIR, which is defined as
\begin{equation}
	h_k \triangleq \sum_{j = 1}^{k}\lVert \epsilon_j \rVert^2_{\varSigma_j^{-1}} - {\sum_{j = 1}^{k}\lVert \epsilon_j^0 \rVert^2_{{(\varSigma^0_j)}^{-1}}}
	\label{hkdefine}
\end{equation}
where $\varSigma^0_j$ is equal to $\varSigma_j$ as mentioned in Remark \ref{rem2}.
\begin{prop}
The GIR  defined in \eqref{hkdefine} can be expressed in terms of $E_k$ and $d_k$ as follows
	\begin{equation}
		h_k = d_k^TE_k^{-1}d_k	\label{hk}
	\end{equation}
\end{prop}

\begin{proof}
	\begin{equation}
		\begin{aligned}
			h_k &= \sum_{j = 1}^{k}\lVert \epsilon_j \rVert^2_{\varSigma_j^{-1}} - {\sum_{j = 1}^{k}\lVert \epsilon_j^0 \rVert^2_{\varSigma_j^{-1}}} \\
			&= \sum_{j = 1}^{k}\epsilon_j^T\varSigma_j^{-1}\epsilon_j \\
			&\quad - \sum_{j = 1}^{k} {\left( \epsilon_j - C_j\varGamma_j\hat{\theta}_k \right)}^T \varSigma_j^{-1} \left( \epsilon_j - C_j\varGamma_j \hat{\theta}_k \right) \\
			&= \sum_{j = 1}^{k} \left( 2\epsilon_j^T\varSigma_j^{-1}C_j\varGamma_j\hat{\theta}_k - \hat{\theta}_k^T\varGamma_j^T C_j^T\varSigma_j^{-1}C_j\varGamma_j\hat{\theta}_k\right) \\
		\end{aligned}
		\label{lk}
	\end{equation}
	where $\theta$ is assumed to be a constant fault vector in \eqref{statefault}, $E_k$ is a symmetric matrix. From \eqref{thetak}, so the \eqref{lk} can be rewritten
	\begin{equation}
		\begin{aligned}
			h_k &= 2\left( \sum_{j = 1}^{k} \epsilon_j^T\varSigma_j^{-1}C_j\varGamma_j \right)E_k^{-1}d_k \\
			&\quad - d_k^TE_k^{-1} \left( \sum_{j = 1}^{k} \varGamma_j^T C_j^T\varSigma_j^{-1}C_j\varGamma_j \right) E_k^{-1}d_k \\
			&= d_k^TE_k^{-1}d_k
		\end{aligned}
	\end{equation}
\end{proof}
It is worth noting that $h_k$ represents the difference between $\sum_{j = 1}^{k}\lVert \epsilon_j \rVert^2_{\varSigma_j^{-1}}$ and $\sum_{j = 1}^{k}\lVert \epsilon_j^0 \rVert^2_{\varSigma_j^{-1}}$. The larger the value of $h_k$, the greater the deviation between the estimated state of the practical system and the estimated state of the fault-free system, indicating a higher likelihood of a fault, and vice versa.


\begin{rem}
	It is important to note that if the $H_{\infty}$ filter is replaced with the Kalman filter, the GIR method will degenerate into the conventional GLR method in \cite{zhang2014}.
	\label{rem:girtoglr}
\end{rem}

\subsection{Fault detection with unknown onset time}
\label{sec:onset}

The fault onset time $r$ is always unknown in the practical faulty system. There is a simple approach \cite{zhang2014} to estimate the onset time $r$. Formulate the $\varGamma_k$ with respect to $r$  
\begin{equation}
	\varGamma_{k+1}(r) = A_k\left( \mathbf{I} - K_kC_k \right)\varGamma_k(r) + \varPsi_k(r)
\end{equation}
where $k$ is the current time step and $\varGamma_0(r) = 0$. And now the $E_k$, $d_k$ and $h_k$ are written as
\begin{align}
	E_k(r) &\triangleq \sum_{j = r + 1}^{k} \varGamma_j^T(r) C_j^T \varSigma_j^{-1} C_j\varGamma_j(r) \\
	d_k(r) &\triangleq \sum_{j = r + 1}^{k} \varGamma_j^T(r) C_j^T \varSigma_j^{-1} \epsilon_j \\
	h_k &= d_k^T(r) E_k^{-1}(r) d_k(r)
\end{align}
Then the estimate of fault onset time $r$ is
\begin{equation}
	\hat{r}_k = \arg\max_{1 \leq r \leq k - s} d_k^T(r) E_k^{-1}(r) d_k(r)
	\label{rk}
\end{equation}

Now, it is feasible to detect fault in LTV system with non-Gaussian noise through \eqref{hk} and estimate fault onset time $r$ through \eqref{rk}. 

\section{Numerical Examples}
\label{sec:example}
Consider the state space system with impulsive fault \cite{zhang2014} modeled by
    \begin{equation}
        x_{k + 1} = \begin{bmatrix}
            \,\,0.5 & 1 \,\,\\
            \,\,0 & 1.2\,\,
        \end{bmatrix}x_k + \begin{bmatrix}
			\,\, 0 \,\,\\
			\,\, 1 \,\,
        \end{bmatrix}u_k + w_k + \delta_{k,\, r}\theta
    \end{equation}

    \begin{equation}
        y_k = \begin{bmatrix}
            \, 1 & 0 \,
        \end{bmatrix}x_k + v_k
    \end{equation}
where $x_k \in \mathbb{R}^2$, $y_k \in \mathbb{R}$, $u_k \in \mathbb{R}$ and the fault vector $\theta \in \mathbb{R}^2$. The covariance of the 1-dimensional output noise $v_k$ is $R_k = 0.0025$ and $v_k \sim \mathcal{N}\left( 0 \,, \,\, R_k\right)$. The form of process noise is as follows
\begin{equation}
	w_k = w_{k - 1} + I_{2 \times 1} \times v_k
\end{equation}
where $I_{2 \times 1} = {\left[\, 1 ,\, 1\,\right]}^T$. It is assumed that $r=201$, meaning the impulsive fault occurs at $k=201$, and the search window $\bar{w} = 100$. The parameter matrices in \eqref{hfkk} are set to $S_k = \mathbf{I}$, $L_k = \mathbf{I}$ and $\alpha = 60$.
\begin{figure}[h]
	\centering
	\includegraphics[width=\hsize]{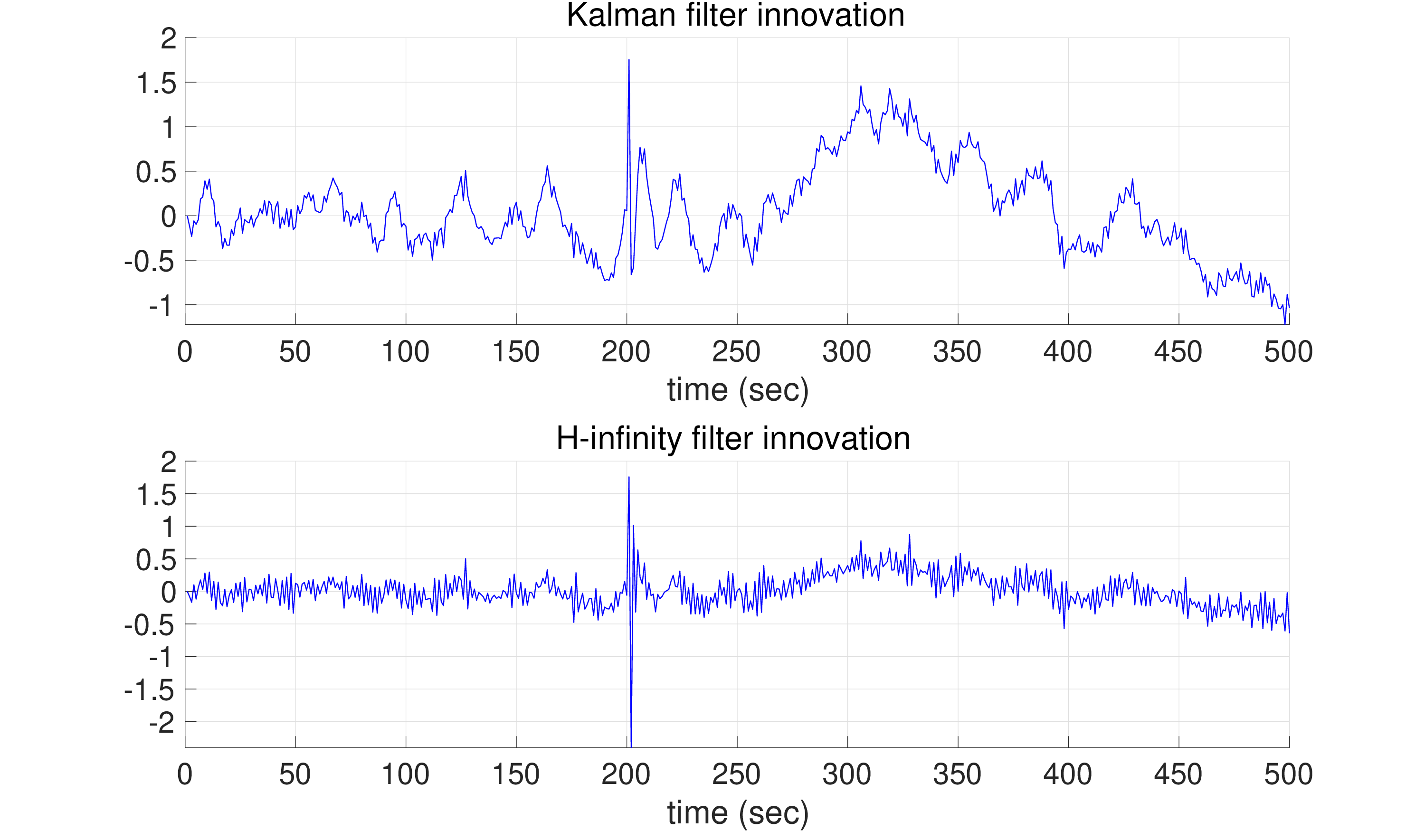}
	\caption{The innovation of Kalman filter (top) and $H_{\infty}$ filter (bottom) for the impulsive fault with $\theta = {\left[\, 1.5 ,\, 0 \,\right]}^T$}
	\label{fig1}
\end{figure}

\begin{figure}[h]
	\centering
	\includegraphics[width=\hsize]{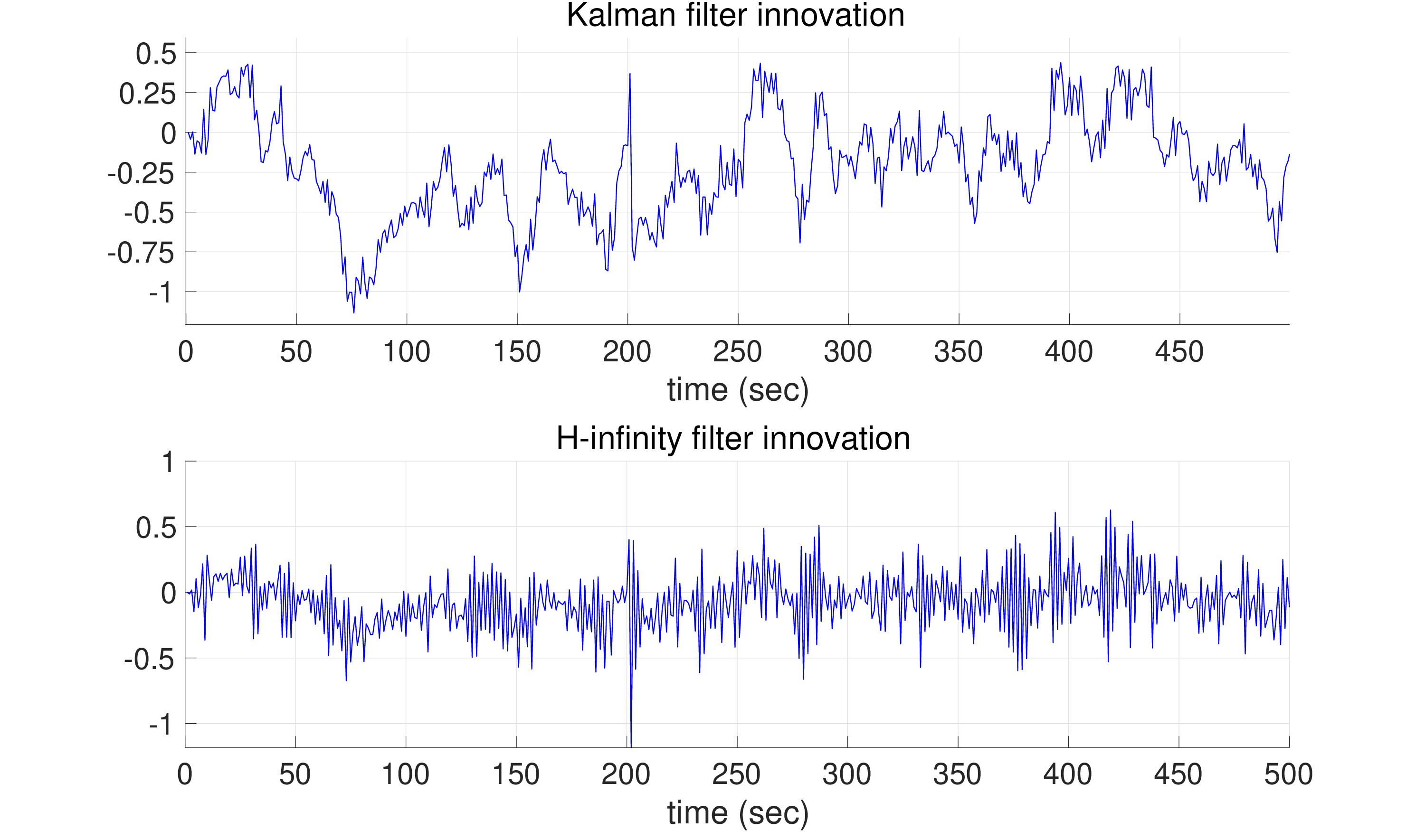}
	\caption{The innovation of Kalman filter (top) and $H_{\infty}$ filter (bottom) for the impulsive fault with $\theta = {\left[\, 0.6 ,\, 0 \,\right]}^T$}
	\label{fig2}
\end{figure}

\begin{figure}[h]
	\centering
	\includegraphics[width=\hsize]{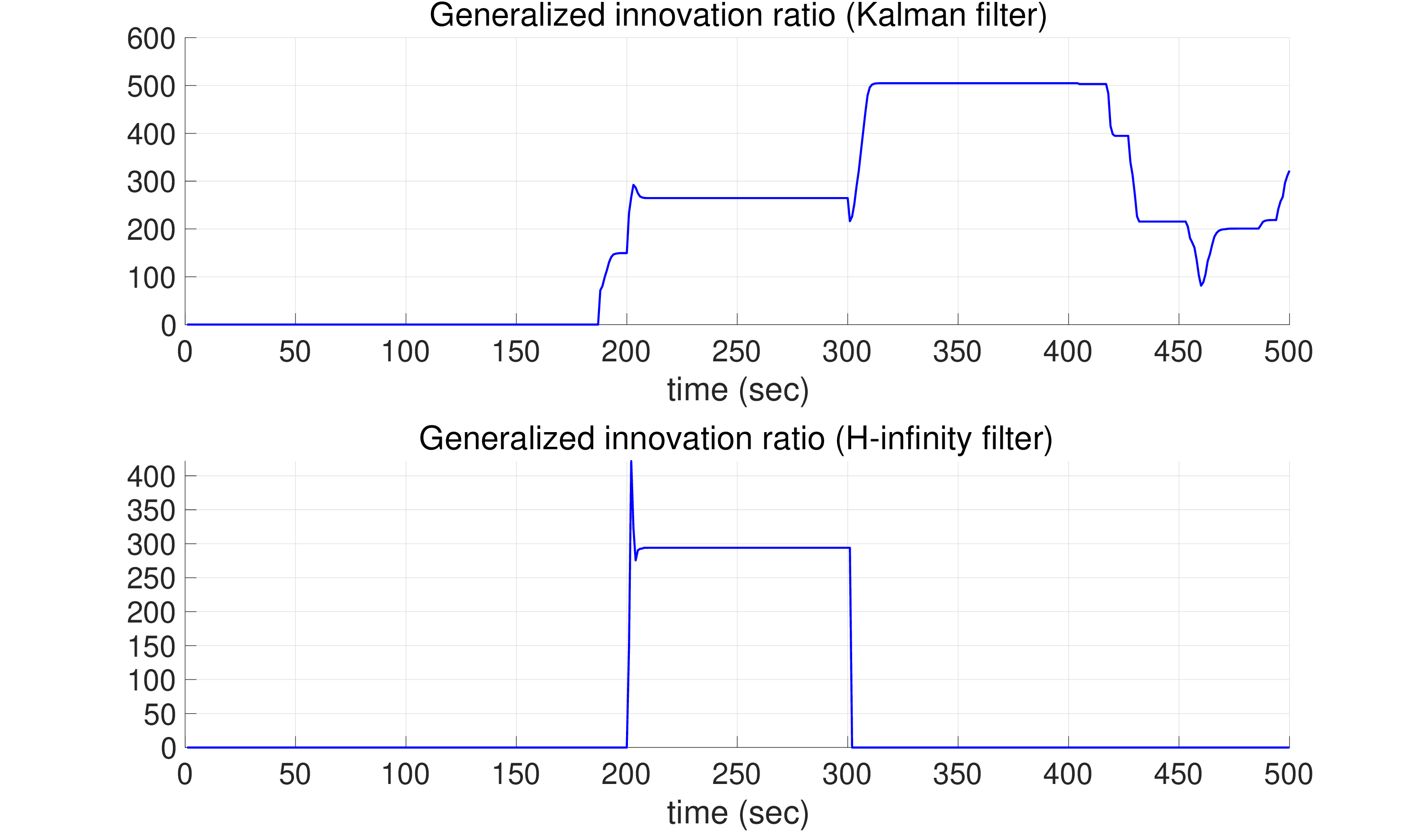}
	\caption{The GIR $h_k$ of Kalman filter (top) and the GIR $h_k$ of $H_{\infty}$ filter (bottom) for the impulsive fault with $\theta = {\left[\, 1.5 ,\, 0 \,\right]}^T$}
	\label{fig3}
\end{figure}




\begin{figure}[h]
	\centering
	\includegraphics[width=\hsize]{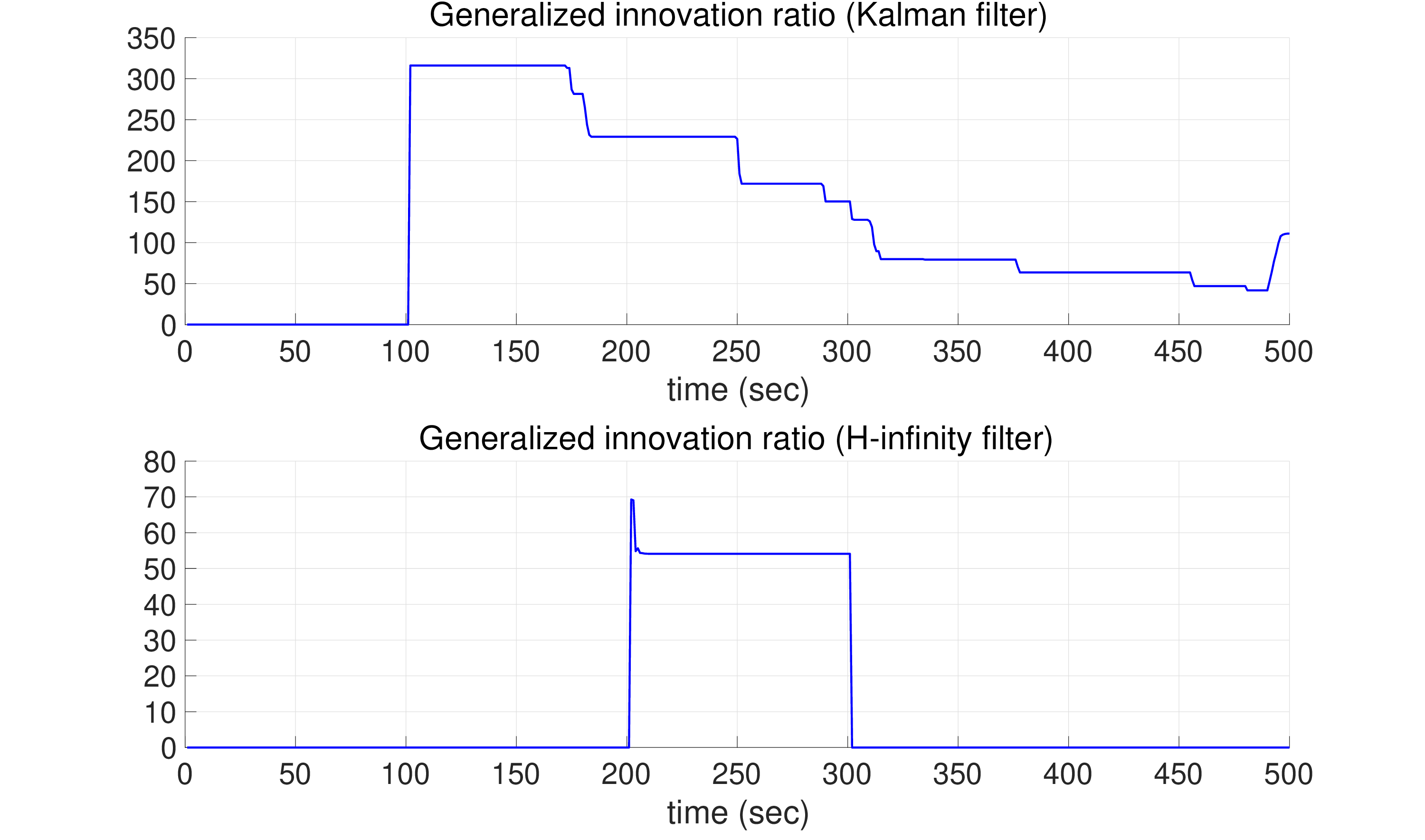}
	\caption{The GIR $h_k$ of Kalman filter (top) and the GIR $h_k$ of $H_{\infty}$ filter (bottom) for the impulsive fault with $\theta = {\left[\, 0.6 ,\, 0 \,\right]}^T$}
	\label{fig4}
\end{figure}


This system is unstable, and a PI controller with the transfer function $G(z) = 0.209 + 0.0011 / \left( z - 1 \right)$ is applied to ensure that the states remain bounded.


Fig. \ref{fig1} depicts the innovations of both the Kalman filter and the $H_{\infty}$ filter with fault onset at $k = 201$ and fault vector $\theta = {\left[\, 1.5 ,\, 0 \,\right]}^T$. It can be observed that the Kalman filter is more sensitive to non-Gaussian noise, resulting in larger fluctuations in its innovations. In this case, the impact of faults on the innovations becomes relatively insignificant. However, for the $H_{\infty}$ filter, faults can be clearly identified through the innovations, which are less affected by the nominal noise. This advantage becomes even more significant when $\theta = {\left[\, 0.6 ,\, 0 \,\right]}^T$, as shown in Fig. \ref{fig2}, where the fault is nearly indistinguishable in the Kalman filter's innovation but still distinguishable in the $H_{\infty}$ filter's innovation.

The above observation is due to the fact that the Kalman filter's performance is limited when applied to systems with non-Gaussian noise, preventing it from providing accurate state estimates. The large estimation errors caused by the nominal noise and the fault are almost indistinguishable. By contrast, the $H_{\infty}$ filter exhibits strong robustness when applied to such systems. This justifies the selection of the $H_{\infty}$ filter for fault detection in systems with non-Gaussian noise.

Figs. \ref{fig3} and \ref{fig4} illustrate the thresholded (where values smaller than the threshold are set to zero) of the GIR $h_k$ for $\theta = {\left[\, 1.5 ,\, 0 \,\right]}^T$ and $\theta = {\left[\, 0.6 ,\, 0 \,\right]}^T$, respectively. At the top of each figure, the results are obtained using the Kalman filter, while at the bottom, the results are obtained using the $H_{\infty}$ filter. 
It can be seen that when the GLR method with the Kalman filter is applied to systems with non-Gaussian noise, fault detection becomes challenging due to large estimation errors. This issue worsens as the fault magnitude decreases. In contrast, the GIR method with the $H_{\infty}$ filter exhibits a clearer pattern for fault detection, even in the presence of small fault magnitudes.

\section{Conclusion}

The solution to the RLS-based $H_{\infty}$ filter is presented first. Then, the fault vector $\theta$ is solved using the least-squares approach. The GIR method, based on the previously developed GLR method and the $H_{\infty}$ filter, is then proposed for fault detection in the presence of non-Gaussian noise. This approach is capable of amplifying abrupt changes (faults) that are difficult to detect using the GLR method. Additionally, due to the robustness of the $H_{\infty}$ filter to variations in system parameters, the GLR method extended with the $H_{\infty}$ filter (referred to as GIR) has demonstrated superior fault detection performance in LTV systems with non-Gaussian noise.



	\end{document}